\documentclass[11pt,english]{smfart}

\usepackage{amsmath,leftidx,amsthm,sfheaders}
\usepackage[all]{xy}
\usepackage{xspace,smfthm}
\usepackage[psamsfonts]{amssymb}
\usepackage[latin1]{inputenc}
\usepackage{graphicx,color}
\usepackage{hyperref,fancyhdr}

\newtheorem*{remark}{\bf Remark}
\newtheorem*{question}{\bf Question}

\newtheorem{theorem}{\bf Theorem}[section]
\newtheorem{proposition}[theorem]{\bf Proposition}
\newtheorem{definition}[theorem]{\bf Definition}
\newtheorem{Theorem}{\bf Theorem}

\newtheorem{lemma}[theorem]{\bf Lemma}

\def\C{{\mathbb C}}
\def\N{{\mathbb N}}

\def\P{\mathbb{P}}

\def\supp{\textup{supp}}

\def\bif{\textup{bif}}

\def\Mand{\mathbf{M}}

\def\ul{{\underline{l}}}

\def\and{{\quad\text{and}\quad}}

\title{The bifurcation measure has maximal entropy}
\author{Henry De Thélin}
\address{LAGA, UMR 7539, Institut Galilée, Université Paris 13, 99 avenue J.B. Clément, 93430 Villetaneuse, France}
\email{dethelin@math.univ-paris13.fr}
\author{Thomas Gauthier}
\address{LAMFA UMR 7352, Universit\'e de Picardie Jules Verne, 33 rue Saint-Leu, 80039 AMIENS Cedex 1, FRANCE}
\email{thomas.gauthier@u-picardie.fr}
\author{Gabriel Vigny}
\address{LAMFA UMR 7352, Universit\'e de Picardie Jules Verne, 33 rue Saint-Leu, 80039 AMIENS Cedex 1, FRANCE}
\email{gabriel.vigny@u-picardie.fr}
\thanks{The second and third authors' research is partially supported by the ANR grant Fatou ANR-17-CE40-0002-01.}

\begin{document}

\begin{abstract} Let $\Lambda$ be a complex manifold and let $(f_\lambda)_{\lambda\in \Lambda}$ be a holomorphic family of rational maps of degree $d\geq 2$ of $\P^1$. We define a natural notion of entropy of bifurcation, mimicking the classical definition of entropy, by the parametric growth rate of critical orbits. We also define a notion a measure-theoretic bifurcation entropy for which we prove a variational principle:  the measure of bifurcation is a measure of maximal entropy. We rely crucially on a generalization of Yomdin's bound of the volume of the image of a dynamical ball.
	
Applying our technics to complex dynamics in several variables, we notably  define and compute the entropy of the trace measure of the Green currents of a holomorphic endomorphism of $\P^k$.
\end{abstract}

\maketitle

\noindent \textbf{Keywords.} Families of rational maps, bifurcation currents and measure, entropy of bifurcation, entropy of rational maps
\medskip

\noindent \textbf{Mathematics~Subject~Classification~(2010):}
37B40, 28D20, 37F45, 37F10.

\tableofcontents

\section{Introduction}
Let $f :\P^k\to \P^k$ be a holomorphic endomorphism of degree $d\geq 2$. The ergodic study of $f$ is well understood: 
\begin{itemize}
	\item  Gromov \cite{gromov_enseignement} showed that the topological entropy of $f$ is $\leq k\log d$. 
\item Forn{\ae}ss and Sibony defined a \emph{Green measure} $\mu_f$ of $f$ as the maximal self-intersection of the \emph{Green current} $T_f$ of $f$. The current $T_f$ is an invariant positive closed current of bidegree $(1,1)$ and mass $1$ whose support is the \emph{Julia set}, the set where the dynamics is chaotic. They showed that $\mu_f$  is mixing (\cite{forsib95}) and has maximal entropy $k\log d$.
\item Briend and Duval then showed that $\mu_f$ is hyperbolic (the Lyapunov exponents are positive) and $\mu_f$ equidistributes the repulsive cycles \cite{briendduval}. Furthermore, $\mu_f$ is the unique measure of maximal entropy \cite{briendduval2}.
\end{itemize}
More generally, for a dominant meromorphic map of a compact Kähler manifold, one want to construct a measure of maximal entropy, to show that it is hyperbolic and that it equidistributes saddle cycles (e.g. for complex Hénon maps, this is done in \cite{BS1, bedfordsmillie3, BLS}). \\

On the other hand,  let now $\Lambda$ be a complex Kähler  manifold and let $\hat{f}:\Lambda \times \P^1\to\Lambda\times\P^1$ be a holomorphic family of rational maps of degree $d \geq 2$: $\hat{f}$ is holomorphic and $\hat{f}(\lambda,z)=(\lambda, f_\lambda(z)) $ where  $f_\lambda$ is a rational map of degree $d$. 
Though the object of study is the notion of $J$-stability, this situation shares many similarity with the iteration of a holomorphic  map of $\P^k$.

Indeed,  DeMarco \cite{DeMarco1} introduced a \emph{current of bifurcation} $T_\bif$ on $\Lambda$, it is a positive closed current of bidegree $(1,1)$ whose support is exactly the unstability locus (the closure of the set where the Julia set does not move continuously) and it is defined has $dd^c L$ where $L$ is the Lyapunov function. Bassanelli and Berteloot \cite{BB1} then defined its self-intersections $T_\bif^l$, the maximal intersection $\mu_{\bif}:=T_\bif^{\dim(\Lambda)}$ is known as the \emph{bifurcation measure}.
 Parallel to the equidistribution of repulsive cycles, several authors have proved various equidistribution properties of specific dynamical parameters towards $\mu_{\bif}$: parameters having a maximal numbers of periodic cycles of given multipliers letting the periods go to $\infty$, strictly post-critically finite parameters letting the preperiods/periods go to $\infty$ (e.g. \cite{FRL, BB2, favregauthier, distribGV, distribGV2, GOV}).

Is it possible to continue the analogy and show that $\mu_{\bif}$ is a measure of maximal entropy? This is the main goal of this paper. Of course, this requires to define a notion of entropy in this situation. \\

To do that, we assume that the family is critical marked: the $2d-2$ critical points can be followed holomorphically (this is always possible up to taking a finite branched cover of $\Lambda$). In other words, there exist holomorphic maps $c_1,\ldots,c_{2d-2} :\Lambda \to \P^1$ with  $f'_\lambda(c_j(\lambda))=0$ and the critical set of $f_\lambda$ is the collection, with multiplicity, $(c_1(\lambda), \dots, c_{2d-2}(\lambda))$.

 For $n \in \N$, we consider the \emph{$n$-bifurcation distance} on $\Lambda$ defined by
$$d_{n}(\lambda,\lambda'):= \max_{1\leq j\leq 2d-2} \max_{0\leq q\leq n-1}  d\left(f^q_\lambda\left(c_{j}(\lambda)\right),f^q_{\lambda'}\left(c_{j}(\lambda')\right)\right),$$
where $d(x,y)$ denotes the Fubini-Study distance on $\P^1$. We say that a set $E \subset \Lambda$ is $(d_{n},\varepsilon)$-separated if :
$$ \min_{\lambda,\lambda'\in E, \ \lambda \neq \lambda'} d_{n}(\lambda,\lambda')\geq \varepsilon. $$
\begin{definition}
	Let $K \subset \Lambda$ be a compact set. We define $h_{\bif}(\hat{f},K)$, the \emph{bifurcation entropy}  of the family $\hat{f}$ in $K$, as the quantity 
	$$h_{\bif}(\hat{f},K):= \lim_{\varepsilon \to 0}  \ \limsup_n \ \frac{1}{n} \log \max \left\{ \mathrm{card}(E), \ E\subset K \ \mathrm{is} \  (d_{n},\varepsilon)-\mathrm{separated} \right\}. $$
	We let  $h_{\bif }(\hat{f}):= \sup_K h_{\bif}(\hat{f},K)$ be the \emph{bifurcation entropy} of the family $\hat{f}$.
\end{definition}
A priori, $h_{\bif}(\hat{f},K) \in [0,+\infty ]$, but refining an argument of \cite{gromov_enseignement} on the growth rate of the volume of the graph, we first show the following bound of the bifurcation entropy outside the support of $T^l_\bif$. 
\begin{Theorem}\label{Majoration}
Pick $1\leq l\leq \dim(\Lambda)$ and $K\Subset\Lambda$. If $K\cap \supp(T_\bif^l)=\varnothing$, then
\[h_\bif(\hat{f},K)\leq (l-1)\cdot \log d.\]
\end{Theorem}
We now want to define a measure-theoretic entropy of bifurcation. The classical definition uses partitions and at some point relies on the invariance of the measure so having a variational principle seems very difficult with that respect. We thus proceed as in \cite{katok1980} and give a definition of (bifurcation) measure-theoretic entropy based on the concept of $(d_{n},\varepsilon)$-separated set.  

Pick a positive Radon measure $\nu$ on $\Lambda$ (for example a probability measure).
\begin{definition}
	 Let $K\subset \Lambda$ be a compact set with $\nu(K)>0$. For any Borel set $X\subset K$ with $\nu(X)<\nu(K)$, let
	 \[h_{\nu,\bif}(\hat{f},K,X):=\lim_{\varepsilon\to0}\limsup_{n\to\infty}\frac{1}{n}\max\left\{\log\mathrm{Card}(E)\, ; \ E\subset X \ \text{is} \ (d_{n},\varepsilon)\text{-separated}\right\}.\]
	 For $0<\kappa < \nu(K)$, we then let:
	 \[h_{\nu,\bif}(\hat{f},K,\kappa):=\inf \{h_{\nu,\bif}(\hat{f},K,X),\  \nu(X)>\nu(K)-\kappa \}.\]
	 We defined the \emph{metric bifurcation entropy} of $\nu $ in $K$, denoted by  $h_{\nu,\bif}(\hat{f},K)$, as 
	 \[h_{\nu,\bif}(\hat{f},K):=\sup_{\kappa\to 0} h_{\nu,\bif}(\hat{f},K,\kappa).\]	 
	 We define the \emph{metric bifurcation entropy} of $\nu $ for the family $\hat{f}$ as
	 \[h_{\nu,\bif}(\hat{f}):= \sup_K h_{\nu,\bif} (\hat{f},K).\]	 
\end{definition}
Observe that $h_{\nu,\bif}(\hat{f}, K\cup K')=\max ( h_{\nu,\bif}(\hat{f}, K),  h_{\nu,\bif}(\hat{f}, K'))$  and that $h_{\nu,\bif}(\hat{f})\leq h_{\bif}(\hat{f})$ for all $\nu$, though there is no natural notion of ergodicity for $\nu$.

Denote by $\mu_{\bif}$ the bifurcation measure of the family $\hat{f}$ (see Section \ref{Section:bifurcationentropy} for a precise definition). We prove the following 
\begin{Theorem}\label{Minoration}For any compact set $K$ such that $\mu_{\bif}(K)>0$ then
	\[ h_{\mu_{\bif}, \bif}(\hat{f},K)= \dim(\Lambda) \log d. \]
	In particular, if $\mu_\bif\neq 0$, one has
	\[h_{\mu_\bif,\bif}(\hat{f})=h_{\bif}(\hat{f})=  \dim(\Lambda) \log d.\]
\end{Theorem}
Notice that the hypothesis $\mu_\bif\neq 0$ is satisfied if and only if there exists a parameter in $\Lambda$ which admits $k$ critical points that are, in a non persistent way, strictly preperiodic to a repelling cycle (\cite{buffepstein,Article1,Dujardin2012}). It is in particular satisfied in any smooth orbifold parametrization of the moduli space of rational maps of degree $d$ with marked critical points. \\

In particular, the theorem asserts that $\mu_\bif$ has maximal entropy in a very strong sense: it only sees sets of maximal entropy and by Theorem~\ref{Majoration}, any compact set outside its support does not carry maximal entropy. This gives a very precise interpretation of the bifurcation measure. A natural question is to know whether any measure satisfying those properties is equivalent to $\mu_\bif$. 

From the two above theorems, we deduce that $\mu_\bif$ satisfies a parametric Brin-Katok formula (see Theorem~\ref{brinkatok}). We show similarly that  the trace measure of $T_\bif^l$ is a measure of maximal entropy in $\mathrm{supp}(T_\bif^l) \backslash \mathrm{supp}(T_\bif^{l+1})$ (see Theorem~\ref{entropy_current_bif}). \\

To prove Theorem~\ref{Minoration}, we use Yomdin's bound of the volume of the image of a dynamical ball \cite{yomdin}. The use of such ideas to compute the entropy of measures in complex dynamics in several variables has been introduced in \cite{bedfordsmillie3}; the first and third authors generalized this idea to give a very general criterion under which we can produce a measure of maximal entropy for a meromorphic map of a compact Kähler manifold (\cite{ThelinVigny1}, a great difficulty arises from the need to control precisely the derivatives near the indeterminacy set). Nevertheless, in both articles, one does not work with the measure directly (but with a Cesàro mean of approximations) and one uses the Misiurewicz's proof of the variational principle to conclude.

In here, the idea is to apply Yomdin's estimate on the parametric balls (with respect to $d_{n}$) directly for $\mu_{\bif}$. We need a precise control on the convergence towards the bifurcation current (we did not have nor needed in \cite{ThelinVigny1} in the general case of meromorphic maps). Our proof leads us to deal with terms of the form
\[ \int_{B_{d_n}\left(x,\varepsilon\right) \cap M} \bigwedge_{j=1}^k (F^{i_j})^*( \Omega) \]
where all $i_j$ are $\leq n-1$, $F$ is a holomorphic map on some manifold $X$ and $M$ is a complex submanifold of $X$ of codimension $k$ endowed with a metric $\Omega$. If all the $i_j$ were either $0$ or $n-1$, this would be the classical Yomdin's bound.  The idea of the proof is to work in the product space $X^k$ and to replace the manifold $M$ with $M^k \cap \Delta$ where $\Delta$ is the diagonal of $X^k$ which still has bounded geometry (Proposition~\ref{alaYomdin}).\\

Going further in the analogy between the dynamics of an endomorphism of $\P^k$ and 
bifurcation in a holomorphic family of rational maps, it is natural to try and define \emph{parametric Lyapunov exponents} by $\chi_j(\lambda)=\lim_n n^{-1} \log |(f_\lambda^n)'(c_j(\lambda))|$ and show that $\chi_j(\lambda)=L(f_\lambda)$ for $\mu_{\bif}$-almost every parameter $\lambda$ (at least in the case of the moduli space of rational maps), where $L(f_\lambda)$ is the Lyapunov exponent of $f_\lambda$ with respect to its unique measure of maximal entropy $\log d$. This has been done successfully in \cite{graczykswiatek} in the very particular case of the unicritic family ($f_\lambda(z)=z^d+\lambda$). The proof relies on subtle properties of external rays and Makarov theorem. Generalizing such result is a challenging question that goes beyond the scope of this article. \\

In a second part of the article, we use the previous technics (especially our variation of Yomdin's estimates) in the case of ergodic theory in several complex variables.

 First, we give an alternate proof of the computation of  the entropy of the Green measure $\mu$ of a Hénon maps (\cite{bedfordsmillie3}). We apply for that our  estimate on the dynamical ball $B_n(x,\varepsilon)$ directly for the measure $\mu$. This allows us to get rid of Misiurewicz' proof of the variational principle (we explain as an application how we can retrieve Brin-Katok formula for Hénon maps). 

Finally, we define a notion of entropy for the trace measure  of the Green currents $T_f^l$ of a holomorphic endomorphism $f:\P^k\to \P^k$ of degree $d$. We use a definition similar to the one we used for bifurcation currents, counting the growth rate of the number of $\varepsilon$-orbit, up to a set of positive  but non total measure. We show that it is always $\geq l \log d$ and that it is equal to $ l \log d$ on compact sets  of $\supp(T^l_f) \backslash \supp (T^{l+1}_f)\neq \varnothing$ having positive trace measure;  nevertheless, we give examples where it is equal to $\alpha$ for any $\alpha\in [l\log d, k \log d]$. This makes the study of the entropy of the trace measure of Green currents richer than the entropy of Green measures (since the latter is always $k \log d$). Finally, note that the idea to do ergodic theory for the trace measure of the Green currents was already exploited in \cite{dujardindirection} where Dujardin defined, for those measures, a notion of Fatou directions, similar to the notion of Lyapunov exponents.  \\

In Section~\ref{preliminaries}, we shall give a general cut-off lemma for dynamical balls and prove our generalization of Yomdin's bound. Then, in Section~\ref{bifurcation}, we recall the construction of the bifurcation currents and properties we need. We then prove Theorems~\ref{Majoration} and \ref{Minoration} in the setting of families of holomorphic endomorphism of $P^q$ with marked points and explain how to get back to the above setting. Finally, Section~\ref{higher} is devoted to our results in complex dynamics in several variables.

\section{Preliminaries}\label{preliminaries}

\subsection{A dynamical cut-off lemma}
Let $X$ be a Kähler manifold endowed with a Kähler form $\Omega$ and let $f:X\to X$ be a holomorphic map.  Let $d$ be the distance associated to $\Omega$. For $n\geq 0$, we have on $X$ the Bowen distance:
\[d_n(x,y):= \max_{i\in \{0, \dots,n-1\}} d(f^i(x),f^i(y)). \]
We denote by $B_{d_n}(x,\varepsilon)$ the ball centered at $x$ and radius $\varepsilon$ for $d_n$. 

Let $Y \Subset X$ be a relatively compact set such that $f(Y)\subset Y$ (if $X$ is compact, one can simply take $Y=X$). Let $\varepsilon_0>0$ be such that $Y_{2\varepsilon_0}$, the $2\varepsilon_0$-neighborhood of $Y$, is still relatively compact in $X$ and $f(Y_{2\varepsilon_0}) \subset Y_{2\varepsilon_0}$.  
\begin{lemma}\label{dynamical_cutoff} We take the above notations. There exists a constant $C$ such that for all $x\in Y$, all $0<\varepsilon<\varepsilon_0$ and all $n \in \N$, there exists a smooth function $\theta_n$ satisfying:
	\begin{itemize}
		\item $\theta_n\equiv 1$ in $B_{d_n}(x,\varepsilon)$ and $\mathrm{supp}(\theta_n)\subset B_{d_n}(x,2\varepsilon)$. 
		\item $C\frac{n^2}{\varepsilon^2} \sum_{i=0}^{n-1} (f^i)^*(\Omega) \pm dd^c \theta_n \geq 0. $
	\end{itemize}
	\end{lemma}
\begin{proof} Using a finite cover, one can construct for every $x\in Y$ a smooth cut-off function $\theta_x$ such that $\theta_x=1$ in $B(x, \varepsilon)$ and $\mathrm{supp}(\theta_x)\subset B(x,2\varepsilon)$ (for the distance $d$). Let $C>0$ be such that for all $x \in Y$, $\varepsilon< \varepsilon_0$:
	\[ \frac{C}{\varepsilon^2} \cdot \Omega \pm dd^c \theta_x \geq 0 \ \mathrm{and} \ d\theta_x \wedge d^c \theta_x \leq \frac{C}{\varepsilon^2} \cdot \Omega .\]
Fix $x\in Y$. We then define $\theta_n:=\Pi_{i=0}^{n-1} \theta_{f^i(x)}\circ f^i $.  By construction, $\theta_n\equiv 1$ in $B_{d_n}(x,\varepsilon)$ and $\mathrm{supp}(\theta_n)\subset B_{d_n}(x,2\varepsilon)$. We compute:
\begin{align*}
dd^c \theta_n =&  \sum_{i=0}^{n-1} \left(\Pi_{j\neq i} \theta_{f^j(x)}\circ f^j\right) dd^c \theta_{f^i(x)}\circ f^i \\
 & + \sum_{\ell \neq \ell'} \left(\Pi_{j\neq \ell, \ j\neq \ell'} \theta_{f^j(x)}\circ f^j\right) d\theta_{f^\ell(x)}\circ f^\ell \wedge  d^c\theta_{f^{\ell'}(x)}\circ f^{\ell'} . 
\end{align*} 
Using that $\pm ( d \psi \wedge d^c \varphi + d \varphi \wedge d^c \psi )\leq d \psi \wedge d^c \psi + d \varphi \wedge d^c \varphi$ and the properties of $\theta_x$ gives
\begin{align*}
0 &\leq \pm dd^c \theta_n + \frac{C}{\varepsilon^2} \sum_{i=0}^{n-1} (f^i)^*(\Omega) + \frac{2C}{\varepsilon^2} \sum_{\ell\neq \ell'}^{n-1} (f^\ell)^*(\Omega)+ (f^{\ell'})^*(\Omega). 
\end{align*} 
The result follows, up to changing the constant $C$. 
\end{proof}

\subsection{A Yomdin's Lemma }
We keep the notations of the above subsection ($f(Y)\subset Y \subset Y_{2\varepsilon_0}\Subset X$). We will need the following variation of Yomdin's bound on the growth of the size of the image of a dynamical ball which uses the Algebraic Lemma (first stated in \cite{yomdin}, see \cite{burguet} for a complete proof). In what follows, we say that a family of smooth manifolds has uniformly bounded geometry if for each $r$, each manifold can be covered by a uniform number of pieces of $C^r$-size equal to $1$. 
\begin{proposition}\label{alaYomdin}  For all $\gamma>0$, there exists $\varepsilon>0$ such that for any family of smooth manifolds $M$ with uniformly bounded geometry and dimension $k$ , there exists an integer $n_0$ such that for any $n\geq n_0$, any  $0\leq i_1  \leq i_2\leq \dots\leq i_k \leq n-1$ and any $x\in Y$, then: 
\[ \int_{B_{d_n}\left(x,\varepsilon\right) \cap M} \bigwedge_{j=1}^k (f^{i_j})^*( \Omega) \leq  e^{\gamma n}\]	
\end{proposition}
 \begin{proof}
 We first briefly recall the strategy of the proof of the bound:
 	 $$\mathrm{Vol}_\Omega(f^{n-1}(B_{d_n}\left(x,\varepsilon\right) \cap M ) ) \leq   e^{\gamma n}.$$
 	  We follow Gromov's exposition \cite{gromov_bourbaki}. Fix some regularity $r \gg 1$ and, up to reducing $M$ assume that its $C^r$-size is $1$ (this is where we use that the geometry is uniformly bounded so that for $n_0$ large enough they are $\leq e^{\gamma n_0}$ pieces). We work in some charts given by a finite atlas. Then, for any unit cubes $\Box_1, \dots, \Box_{i}$, let $M_i:= f^i \left(M \cap f^{-1}(\Box_1) \cap \dots \cap  f^{-i}(\Box_i)\right)$, then the Algebraic Lemma implies (see \cite{gromov_bourbaki}[(*) p 233]):
 \begin{equation}\label{etoile}
  \mathrm{Vol}_\Omega(M_i) \leq (C \|D_rf\|^{\frac{2k}{r}}  +1 )^i
 \end{equation}
 	where $C$ depends on $r$, the real dimensions of $M$ ($=2k$) and $X$ but not on $f$ nor $M$ and where $\|D_rf\|$ is the supremum of the derivatives of all order $\leq r$. 
 	
 	Take some $j \geq 1$ and some dynamical ball $B_{d_n}(x,1/j)$. We take $1 \ll m \ll n-1$. Let us assume to simplify that $n-1=mi$ for some $i\in \N$. Consider the $1/j$-cubes $\tilde{\Box}_l$ centered at $f^{ml}(x)$ for all $1\leq l\leq i$. Then 
 	\[f^{n-1}(B_{d_n}\left(x,1/j\right) \cap M) \subset (f^m)^{i}(M \cap  (f^m)^{-1}(\tilde{\Box}_1) \cap \dots \cap  (f^m)^{-i}(\tilde{\Box}_i)).\]  
Rescaling to $1$-cubes, we deduce from \eqref{etoile} that:
  	\[  \mathrm{Vol}_\Omega(f^{n-1}(B_{d_n}\left(x,1/j\right) \cap M)) \leq   (C \|D_rf_j^m\|^{\frac{2k}{r}}  +1 )^i,\]
  	where $f_j$ is the rescaled map in each unit cube $\Box_l:= j.\tilde{\Box}_l$
  	  	so $f_j(t):= jf(j^{-1} t)$ (working in some charts). We choose $j$ large enough so that $\|D_rf_j^m\| \leq 2 \|Df^m\|$ (rescaling reduces the norm of the derivatives of order $>1$ and has no effect on order $1 $). In particular:
  	  	\[   \mathrm{Vol}_\Omega(f^{n-1}(B_{d_n}\left(x,1/j\right) \cap M)) \leq   (C' \|Df^m\|^{\frac{2k}{r}}+1 )^i,\]
where $C'$ is another constant that depends only on $r$, the dimensions of $M$ and $X$ (we assume that $\|Df\| \geq 1$, if not then the result is already obvious). In particular, using $\|Df^m\| \leq \|Df\|^m$, $n-1=mi$, we recognize:
 	\[   \mathrm{Vol}_\Omega(f^{n-1}(B_{d_n}\left(x,1/j\right) \cap M)) \leq   \left((C')^\frac{1}{m}\right)^{n-1}.( \|Df\|^{\frac{2k}{r}})^n.\]
In fact, $r$ was chosen so that $\|Df\|^{\frac{2k}{r}}\leq e^{\gamma/2}$ and we now choose $m$ large enough so that $(C')^\frac{1}{m} \leq e^{\gamma/2}$ which proves the result (if $n-1\neq im$ we simply prove the bound for $n'=i m$ and we have a extra $\|Df\|^m$ that appears). \\

We now prove the Proposition. We take $\varepsilon\leq \varepsilon_0$, in what follows, the norms are taken on $Y_{2\varepsilon_0}$ (we are only interested in points whose orbit stays in $B_{d_n}\left(x,\varepsilon\right)\subset Y_{2\varepsilon_0}$). Observe that if all $i_j$ are equal to $n-1$, then the proposition means that:
\[ \int_{B_{d_n}\left(x,\varepsilon\right) \cap M} (f^{n-1})^*( \Omega^k)= \mathrm{Vol}_\Omega(f^{n-1}(B_{d_n}\left(x,\varepsilon\right) \cap M ) ) \leq   e^{\gamma n} \]
which is the classical bound. Similarly if $i_j=0$ for $j\leq j_0$  and $i_j=n-1$ for $j>j_0$, then we can bound, near $x$,  $\bigwedge_{j\leq j_0} \Omega$ by a finite sum of currents of integration on lamination by linear subspaces of codimension $j_0$:
\[ \bigwedge_{j\leq j_0} \Omega \leq C \sum_{\alpha} \int_{\alpha} [L_\alpha(u)] d\lambda_\alpha(u)     \]  
where $C$ is a constant that depends (locally uniformly) in $x$, the $\alpha$ are the subspaces of dimension $j_0$ given by the coordinates (we work in some chart), $L_\alpha(u)$ is the  subspace of codimension $j_0$ directed by the remaining coordinates that intersects $\alpha$ at $u$ and $\lambda_\alpha$ is the Lebesgue measure on $\alpha$. In particular, it is sufficient to bound the term
\[ \int_{B_{d_n}\left(x,\varepsilon\right) \cap M} (f^{n-1})^*( \Omega^{k-j_0}) \wedge [L(u_1)\cap \dots \cap L(u_{j_0})]  \]
uniformly in $u_1, \dots u_{j_0}$. Since it can be rewritten as:
\[ \mathrm{Vol}_\Omega\left(f^{n-1}(B_{d_n}\left(x,\varepsilon\right) \cap M \cap L(u_1)\cap \dots \cap L(u_{j_0})) \right)   \]
and $M \cap L(u_1)\cap \dots \cap L(u_{j_0}))$ has uniformly bounded geometry, the wanted inequality is again the classical Yomdin's result. \\

Let $\gamma  >0 $. Fix $\iota\ll 1$, independent of $n$,  and let $m= E(\iota n )$ ($E$ being the integer part). For all $i_j$, write $i_j= l_j m + r_j$ with $0\leq r_j < m$.  Then, as $f^*(\Omega) \leq \|Df\|^2 \Omega$ (up to changing the norm $\|Df\|$ by a constant), we have:
\[ \int_{B_{d_n}\left(x,\varepsilon\right) \cap M} \bigwedge_{j=1}^k (f^{i_j})^*( \Omega) \leq  \|Df\|^{2km}  \int_{B_{d_n}\left(x,\varepsilon\right) \cap M} \bigwedge_{j=1}^k (f^{l_jm})^*( \Omega). \]
In particular, $\|Df\|^{2km} \leq (\|Df\|^{2k\iota})^n \leq e^{\gamma n}$ by taking $\iota$ small enough which we do. 

 Assume that $l_j=0$ for $j\leq j_0$ Proceeding as above, we can replace $\Omega$ by  a finite sum of currents of integration on lamination by linear subspaces of codimension $j_0$. So we are reduced to the case of terms of the form $ \int_{B_{d_n}\left(x,\varepsilon\right) \cap M \cap L} \bigwedge_{j_0< j \leq k } (f^{i_j})^*( \Omega)$ where $L$ is a linear subspace, so  it is the same estimate with $M$ replaced by $M\cap L$. 
 
 Let us thus assume that $l_j \neq 0$ for all $j$. Write $\underline{l}:=(l_1,\dots, l_k)$. Let $f_{\underline{l}}:= (f^{l_1}, \dots, f^{l_k}): X^k \to X^k$  and $\Delta$ be the diagonal in $X^k$:
 \[\Delta:= \{(x,\dots, x) \in X^k\}. \] 
Let $\delta_k$ be the product distance on $X^k$: $\delta_k((x_i)_{i\leq k}, (y_i)_{i\leq k}):= \max_i d(x_i, y_i)$ and let $d_{k,p, f_\ul}$ be the Bowen distance in $X^k$ for the $p-1$-iterate associated to $f_\ul$. We let $B_{d_{k,p, f_\ul}}( (x_i)_{i\leq k}, \varepsilon)$ be the associated Bowen ball. Finally, let $\Omega_k:= \sum_j \pi_j^*\Omega$ where $\pi_j$ is the projection from $X^k$ to the $j$-th factor and $\tilde{M}:= M^k \cap \Delta$.  With these notations, we have:
\begin{align*}
 \int_{B_{d_n}\left(x,\varepsilon\right) \cap M} \bigwedge_{j=1}^k (f^{l_jm})^*( \Omega)  &\leq \int_{ \left( \pi_1^{-1}(B_{d_n} \left(x,\varepsilon\right))\right) \cap  \tilde{M} } \bigwedge_{j=1}^k \pi_j^*((f^{l_j m})^*( \Omega))\\
            &\leq   \int_{ B_{d_{k,m, f_\ul}}( (x, x, \dots ,x), \varepsilon) \cap \tilde{M} }  (f_\ul^m)^*(\Omega^k_k). 
\end{align*}
By the above proof of the bound
 	 $$\mathrm{Vol}_\Omega(f^{n-1}(B_{d_n}\left(x,\varepsilon\right) \cap M ) ) \leq   e^{\gamma n},$$
 	where we replace $M$ by $\tilde{M}$, $f$ by $f_\ul$ and $n$ by $m$, we infer that:
 \begin{align*}
 \int_{ B_{d_{k,m, f_\ul}}( (x, x, \dots ,x), 1/j) \cap \tilde{M} }   (f_\ul^m)^*(\Omega^k_k) &\leq  (C'' \| Df_\ul \|^{\frac{\tilde{2k}}{r}})^m\\
  &\leq (C'')^{\iota n} \| Df\|^\frac{n}{r} 
 \end{align*}	 
 where $C''$ is a constant that depends on $r$, the dimension $2k$ of $\tilde{M}$ and the dimension of $X^k$. Observe that it is crucial that, since $\iota$ is fixed, there is only finitely many $\ul$ (roughly $\leq \iota^{-k}$) so we can find a $j$ for which $\| D_r f_{\ul,j}  \| \leq 2 \|Df_\ul \|$ for all $\ul$ simultaneously. We conclude as above by taking $r$ large enough and adding the constraint $(C'')^{\iota} \leq e^\gamma$. 
 \end{proof}

 \section{Bifurcation entropy}\label{bifurcation}
 \subsection{Background in bifurcation theory}\label{background}
 \subsubsection{Defining the bifurcation currents and (locally uniform) estimates}
For this section, we follow the presentation of \cite{favredujardin,Dujardin2012}. Even though everything is presented in the case $q=1$ and for marked \emph{critical} points, the exact same arguments give what present below.

\smallskip

Let $\Lambda$ be a complex manifold and let $\hat{f}:\Lambda \times \P^q\to\Lambda\times\P^q$ be a holomorphic family of endomorphisms of $\P^q$ of algebraic degree $d \geq 2$: $\hat{f}$ is holomorphic and $\hat{f}(\lambda,z)=(\lambda, f_\lambda(z)) $ where  $f_\lambda$ is an endomorphism of $\P^q$ of algebraic degree $d$.

Let $\omega_{\P^q}$ be the standard Fubini-Study form on $\P^q$ and $\pi_\Lambda:\Lambda\times\P^q\to\Lambda$ and $\pi_{\P^q}:\Lambda\times\P^q\to\P^q$ be the canonical projections. Finally, let $\widehat{\omega}:=(\pi_{\P^q})^*\omega_{\P^q}$. It is known that
the sequence $d^{-n}(\widehat{f}^n)^*\widehat{\omega}$ converges to a closed positive $(1,1)$-current $\widehat{T}$ on $\Lambda\times\P^q$ with continuous potential. Moreover, for any $1\leq j\leq q$, it satisfies 
\[\widehat{f}^*\widehat{T}^j=d^j\cdot \widehat{T}\]
and $\widehat{T}^q|_{\{\lambda_0\}\times\P^1}=\mu_{\lambda_0}$ is the unique measure of maximal entropy $q\log d$ of $f_{\lambda_0}$ for all $\lambda_0\in\Lambda$.

For any $n\geq1$, we have $\widehat{T}=d^{-n}(\hat{f}^{n})^*\hat{\omega}+d^{-n}dd^c\widehat{u}_n$, where $(\widehat{u}_n)_n$ is a locally uniformly bounded sequence of continuous functions.

\medskip

Assume now that the family $\hat{f}$ is endowed with $k$ marked points i.e. we are given holomorphic maps $a_1,\ldots,a_k :\Lambda \to \P^q$. Let $\Gamma_{a_j}$ be the graph of the map $a_j$ and set
 \[\mathfrak{a}:=(a_1,\ldots,a_k).\]
 \begin{definition}
For $1\leq i\leq k$, the \emph{bifurcation current} $T_{a_i}$ of the point $a_i$ is the closed positive $(1,1)$-current on $\Lambda$ defined by
 \[T_{a_i}:=(\pi_{\Lambda})_*\left(\widehat{T}\wedge[\Gamma_{a_j}]\right)\]\\
 and we define the \emph{bifurcation current} $T_{\mathfrak{a}}$ of the $k$-tuple $\mathfrak{a}$ as
 \[T_{\mathfrak{a}}:=T_{a_1}+\cdots+T_{a_k}.\]
 \end{definition}

For any $\ell\geq0$, write
\[\mathfrak{a}_\ell(\lambda):=\left(f_\lambda^{\ell}(a_1(\lambda)),\ldots,f_\lambda^{\ell}(a_k(\lambda))\right), \ \lambda\in\Lambda.\]
Let now $K\Subset\Lambda$ be a compact subset of $\Lambda$ and let $\Omega$ be some compact neighborhood of $K$, then $(a_{\ell})^*(\omega_{\P^{q}})$ is bounded in mass in $\Omega$ by $C d^{\ell}$, where $C$ depends on $\Omega$ but not on $\ell$.

Applying verbatim the proof of \cite[Proposition-Definition~3.1]{favredujardin}, we have the following

\begin{lemma}\label{lm:DF}
For any $1\leq i\leq k$, the support of $T_{a_i}$ is the set of parameters $\lambda_0\in\Lambda$ such that the sequence $\{\lambda\mapsto f_\lambda^n(a_i(\lambda))\}$ is not a normal family at $\lambda_0$.

Moreover, writing $a_{i,\ell}(\lambda):=f_{\lambda}^\ell(a_i(\lambda))$, there exists a locally uniformly bounded family $(u_{i,\ell})$ of continuous functions on $\Lambda$ such that 
 \[(a_{i,\ell})^*(\omega_{\P^{q}})=d^{\ell} T_{a_i}+dd^cu_{i,\ell}, \ \text{on} \ \Lambda.\]
\end{lemma}

 \subsubsection{Higher bifurcation currents of marked points}
As the convergence in Lemma~\ref{lm:DF} holds uniformly on compact subsets of $\Lambda$, for all $j\geq1$, we have
\[(a_{i,\ell})^*(\omega_{\P^{q}}^j)=d^{j\ell} T_{a_i}^j+dd^cO(d^{(j-1)\ell})\]
on compact subsets of $\Lambda$. In particular, one sees that
\begin{align}
T_{a_i}^{q+1}=0 \ \text{on} \ \Lambda,\label{nointersection>q}
\end{align}
and, for any $1\leq j\leq kq$, this easily gives
\[T_\mathfrak{a}^{j}=\sum_{m=1}^j\sum T_{a_{r_1}}^{\alpha_1}\wedge \cdots\wedge T_{a_{r_m}}^{\alpha_m},\]
where the second sum ranges over all $m$-tuples $(r_1,\ldots,r_m)$ of indices with $1\leq r_\ell\leq k$ and all $m$-tuples $A_m=(\alpha_1,\ldots,\alpha_m)$ with $1\leq \alpha_\ell\leq q$ and $\sum_\ell\alpha_\ell=j$, see e.g. \cite{Article1,Dujardin2012} for the case $q=1$. 

\medskip 
 
Let us still denote  $\pi_\Lambda:\Lambda \times(\P^q)^k\to \Lambda$ be the projection onto the first coordinate and for $1\leq i\leq k$, let $\pi_i:\Lambda \times(\P^q)^k\to \P^q$ be the projection onto the $i$-th factor of the product $(\P^q)^k$.  
Finally, we denote by $\Gamma_{\mathfrak{a}}$ the graph of $\mathfrak{a}$:
 \[\Gamma_{\mathfrak{a}}:= \{ (\lambda, (z_j) ), \ \forall j, z_j= a_j(\lambda)\} \subset \Lambda \times (\P^q)^k.\] 
Following verbatim the proof of~\cite[Lemma 2.6]{AGMV}, for any $i$, we get
\[T_{\mathfrak{a}}^{i}=\sum_{j_1,\ldots,j_k\geq0 \atop j_1+\cdots+j_k=i}\bigwedge_{\ell=1}^kT_{a_\ell}^{j_\ell}= \sum_{j_1,\ldots,j_k\geq0 \atop j_1+\cdots+j_k=i}(\pi_\Lambda)_* \left( \bigwedge_{\ell=1}^k \pi_\ell^*\left(\widehat{T}^{j_\ell}\right) \wedge \left[\Gamma_{\mathfrak{a}}\right]\right).\]


\subsection{Bifurcation Entropy for a $k$-tuple of marked points}

 We now assume that $\Lambda$ is a complex Kähler manifold endowed with a Kähler form $\omega_\Lambda$ and that the family $\hat{f}$ comes with $k$ marked points. Recall that we have set
 \[\mathfrak{a}:=(a_1,\ldots,a_k).\]
 In analogy with topological entropy, we define a notion of \emph{bifurcation entropy} of the marked family $(\hat{f},\mathfrak{a})$ in the following way. For $n \in \N$, we consider the \emph{$n$-bifurcation distance} on $\Lambda$ associated with $\mathfrak{a}$ defined by
$$d_{\mathfrak{a},n}(\lambda,\lambda'):= \max_{1\leq j\leq k} \max_{0\leq i\leq n-1}  d\left(f^i_\lambda\left(a_j(\lambda)\right),f^i_{\lambda'}\left(a_j(\lambda')\right)\right),$$
where $d(x,y)$ denotes the Fubini-Study distance on $\P^1$. We say that a set $E \subset \Lambda$ is $(d_{\mathfrak{a},n},\varepsilon)$-separated if :
$$ \min_{\lambda,\lambda'\in E, \ \lambda \neq \lambda'} d_{\mathfrak{a},n}(\lambda,\lambda')\geq \varepsilon. $$
\begin{definition}
Let $K \subset \Lambda$ be a compact set, we define the \emph{bifurcation entropy} $h_\mathfrak{a}(\hat{f},K)$ of the marked family $(\hat{f},\mathfrak{a})$ in $K$ as the quantity:
$$h_\mathfrak{a}(\hat{f},K):= \lim_{\varepsilon \to 0}  \ \limsup_n \ \frac{1}{n} \log \max \left\{ \mathrm{card}(E), \ E\subset K \ \mathrm{is} \  (d_{\mathfrak{a},n},\varepsilon)-\mathrm{separated} \right\}. $$
\end{definition}
A priori, $h_\mathfrak{a}(\hat{f},K) \in [0,+\infty ]$, but using the fact that, by definition, the bifurcation entropy is bounded by the topological entropy on $K \times (\P^q)^k$ of $\hat{f}_k:\Lambda \times (\P^q)^k \to \Lambda \times (\P^q)^k$ defined by $\hat{f}_k(\lambda,z_1,\dots z_k)=(\lambda, f_\lambda(z_1),\dots, f_\lambda(z_k))$, we have $h_\mathfrak{a}(\hat{f},K) \leq k q\log d$ (see \cite{gromov_enseignement} or the proof below). Observe that the entropy of $\hat{f}_k$ is already  $k q\log d$ on each (invariant set) $\{\lambda\} \times (\P^q)^k$ so it is not related to bifurcation phenomena. We show here that bifurcation entropy can be bounded from above outside the support of $T^i_\mathfrak{a}$. The proof follows the idea of the first author (\cite{dethelin_mathann}, see also \cite{dinh_JGA}):
 
\begin{theorem}\label{alaGromov} Pick $1\leq i\leq \dim(\Lambda)$. Assume that $K\cap \supp(T_{\mathfrak{a}}^{i})=\varnothing$. Then
		\[h_\mathfrak{a}(\hat{f},K)\leq (i-1)\log d.\]
\end{theorem}
\begin{proof}
Fix $n\geq1$ and $\varepsilon>0$.
Let $p_j:(\P^q)^k\to\P^q$ be the projection onto the $j$-th coordinate. We define a K\"ahler metric on $(\P^q)^k$ by letting $\omega:=p_1^*(\omega_{\P^q})+\cdots+p_k^*(\omega_{\P^q})$,
here $\omega_{\P^q}$ is the Fubini-Study metric on $\P^q$ of mass $1$.

Let $\omega_\Lambda$ be a K\"ahler metric on $\Lambda$ and $\omega_\ell:=P^*_\ell(\omega)$, where $P_\ell:\left((\P^q)^{k}\right)^n\to(\P^q)^k$ is the projection onto the $\ell$-th factor of the form $(\P^q)^k$.  If $\Pi_\Lambda:\Lambda\times(\P^q)^{kn}\to\Lambda$ is the canonical projection onto $\Lambda$, we still denote by $\omega_\Lambda$ the pull-back  $\Pi_\Lambda^*(\omega_\Lambda)$. We endow $\Lambda\times(\P^q)^{kn}$ with the product K\"ahler metric 
\[\Omega:=\omega_\Lambda+\sum_{\ell=1}^{n}\omega_\ell\]
and denote by $\tilde{d}$ the induced distance on $\Lambda\times(\P^q)^{kn}$.

As before, set $\mathfrak{a}_\ell(\lambda):=(f_\lambda^\ell(a_1(\lambda)),\ldots,f_\lambda^\ell(a_k(\lambda)))$ for all $\lambda\in\Lambda$ and all $\ell\geq0$.
Let $\Gamma_n\subset\Lambda\times (\P^q)^{kn}$ be the graph of the map $\mathfrak{A}^n:=(\mathfrak{a}_0,\ldots,\mathfrak{a}_{n-1}):\Lambda\to (\P^q)^{kn}$, and pick a set $E\subset \Lambda$ which is $(d_{\mathfrak{a},n},\varepsilon)$-separated and let $N:=\mathrm{Card}(E)$. For $\lambda \in E$, let $\tilde{\lambda}:=\mathfrak{A}^n(\lambda)$. If $\tilde{E}:=\mathfrak{A}^n(E)$, we have $\tilde{d}(\tilde{\lambda}_1,\tilde{\lambda}_2) \geq\varepsilon$,
for any distinct $\tilde{\lambda}_1,\tilde{\lambda}_1\in \tilde{E}$. In particular, if $K_\varepsilon:=\{\lambda\in\Lambda\, ; d_\Lambda(\lambda,K)\leq\varepsilon\}$, we have
\[\bigcup_{\tilde{\lambda}\in\tilde{E}}B_{\tilde{d}}\left(\tilde{\lambda},\frac{\varepsilon}{2}\right)\cap\Gamma_n\subset\Pi^{-1}_\Lambda(K_\varepsilon)\cap\Gamma_n,\]
and the union is disjoint. So
\begin{align*}
\mathrm{Vol}_\Omega\left(\Pi^{-1}_\Lambda(K_\varepsilon)\cap\Gamma_n\right) & \geq  \sum_{\tilde{\lambda}\in\tilde{E}}\mathrm{Vol}_\Omega\left(B_{\tilde{d}}\left(\tilde{\lambda},\frac{\varepsilon}{2}\right)\cap\Gamma_n\right).
\end{align*}
Since $\Gamma_n$ is an analytic subvariety of $\Lambda\times (\P^q)^{kn}$ of complex dimension $\dim(\Lambda)$ passing through the center of the balls $B_{\tilde{d}}(\tilde{\lambda},\varepsilon/2)$,  Lelong's inequality implies
\[\mathrm{Vol}_\Omega\left(B_{\tilde{d}}\left(\tilde{\lambda},\frac{\varepsilon}{2}\right)\cap\Gamma_n\right) \geq  c\varepsilon^{2\dim(\Lambda)} \]
where $c$ is a constant that does not depend on $n$. Since $\mathrm{Card}(\tilde{E})=\mathrm{Card}(E)=N$, we get
\begin{align}
\mathrm{Vol}_\Omega\left(\Pi^{-1}_\Lambda(K_\varepsilon)\cap\Gamma_n\right) & \geq N\cdot c\varepsilon^{2\dim(\Lambda)}.\label{eq:volume}
\end{align}

We now bound $\mathrm{Vol}_\Omega\left(\Pi^{-1}_\Lambda(K_\varepsilon)\cap\Gamma_n\right)$ from above. Let $A_s$ denote the set of $\alpha:= (\alpha_1,\ldots,\alpha_k)$ such that $0\leq \alpha_j\leq q$ and $\sum \alpha_j = s$. Let $L_n$ be the set of $k$-tuples $\ell=(\ell_1,\dots,\ell_k)$ of distinct integers in $\{1,\dots, n \}$. Note that the cardinality of 
$A_s$ is $\leq (q+1)^s$ and the cardinality of $L^s_n$ is $\leq n^k$.

Write $d_\Lambda:=\dim(\Lambda)$. Up to reducing $\varepsilon>0$, we can assume $K_{(d_\Lambda+1)\varepsilon}\cap \supp(T_{a_1}^{\alpha_1}\wedge \cdots \wedge T_{a_{k}}^{\alpha_k})=\varnothing$ for all $\alpha \in A_i$. Choose $\mathcal{C}^2$ non-negative functions $\theta_1,\ldots,\theta_i$ on $\Lambda$ such that $\theta_j\equiv1$ on $K_{j\varepsilon}$ and $\supp(\theta_j)\subset K_{(j+1)\varepsilon}$ for all $1\leq j\leq i$. We then have 
\begin{align*}
\mathrm{Vol}_\Omega\left(\Pi^{-1}_\Lambda(K_\varepsilon)\cap\Gamma_n\right) & \leq  \int_{\Lambda\times(\P^q)^{kn}} \left(\theta_1\circ\Pi_\Lambda\right) \, \Omega^{d_\Lambda}\wedge[\Gamma_n]\\
& \leq \int_{\Lambda}\theta_1 \cdot(\Pi_\Lambda)_*\left(\sum_{s=0}^{d_\Lambda} \binom{d_\Lambda}{s} \left(\omega_\Lambda\right)^{d_\Lambda-s}\wedge\sum_{\alpha \in A_s}\sum_{\ell \in L_n} \bigwedge_{j=1}^s \omega_{\ell_j}^{\alpha_j}\wedge[\Gamma_n] \right)\\
& \leq \int_{\Lambda}\theta_1 \left( \sum_{s=0}^{d_\Lambda} \binom{d_\Lambda}{s} \omega_\Lambda^{d_\Lambda-s}\wedge\sum_{\alpha\in A_s} \sum_{\ell \in L_n} \bigwedge_{j=1}^k(\mathfrak{a}_{\ell_j})^*(\omega^{\alpha_j}) \right).
\end{align*}
Fix an integer $s \leq d_\Lambda$, a $k$-tuple $\alpha \in A_s$ and a $k$-tuple $\ell \in L_n$. Recall that, by definition, we have $T_\mathfrak{a}=\sum_{j\leq k}T_{a_j}$. As seen in Section~\ref{background}, there exists a locally uniformly bounded family $(u_\ell)$ of continuous functions on $\Lambda$ such that 
\[(\mathfrak{a}_\ell)^*(\omega)=d^\ell T_{\mathfrak{a}}+dd^cu_\ell\]
for all $\ell\geq0$ and that $(\mathfrak{a}_\ell)^*(\omega^j)=((\mathfrak{a}_\ell)^*(\omega))^j$ for all $1\leq j\leq q$. Assume for simplicity that $\alpha_1 \neq 0$. Then, letting $S=\bigwedge_{j=2}^k(\mathfrak{a}_{\ell_j})^*(\omega^{\alpha_j})$, by Stokes  and using $\theta_{2}\equiv1$ on $\supp(\theta_1)$:
\begin{align*}
\int_\Lambda\theta_1 \cdot\omega_\Lambda^{d_\Lambda-s}\wedge \bigwedge_{j=1}^k(\mathfrak{a}_{\ell_j})^*(\omega^{\alpha_j}) 
=& \int_\Lambda\theta_1 \cdot\omega_\Lambda^{d_\Lambda-s}\wedge \left(d^{\ell_1} T_{\mathfrak{a}}+dd^cu_{\ell_1} \right)^{\alpha_1} \wedge S\\
 =& \int_\Lambda\theta_1 \cdot\omega_\Lambda^{d_\Lambda-s}\wedge d^{\ell_1} T_{\mathfrak{a}}\wedge \left(d^{\ell_1} T_{\mathfrak{a}}+dd^cu_{\ell_1} \right)^{\alpha_1-1}\wedge S \\
+& \int_\Lambda  u_{\ell_1} dd^c(\theta_1) \wedge \omega_\Lambda^{d_\Lambda-s}\wedge\left(d^{\ell_1} T_{\mathfrak{a}}+dd^cu_{\ell_1} \right)^{\alpha_1-1}\wedge S \\
\leq& \, d^n \int_\Lambda\theta_1 \cdot\omega_\Lambda^{d_\Lambda-s}\wedge  T_{\mathfrak{a}}\wedge\left(d^{\ell_1} T_{\mathfrak{a}}+dd^cu_{\ell_1} \right)^{\alpha_1-1}\wedge S \\
+& C \int_\Lambda \theta_2  \omega_\Lambda^{d_\Lambda-s+1}\wedge \left(d^{\ell_1} T_{\mathfrak{a}}+dd^cu_{\ell_1} \right)^{\alpha_1-1}\wedge S, 
\end{align*}
where $C$ is a constant that depends on the $\mathcal{C}^2$-norm of $\theta_1$ and the supremum of the $L^\infty$-norm of the $(u_\ell)$ but not on $n$. Iterating the process, we get the bound:
\begin{align*}
\int_\Lambda\theta_1 \cdot\omega_\Lambda^{d_\Lambda-s}\wedge \bigwedge_{j=1}^s(\mathfrak{a}_{\ell_j})^*(\omega^{\alpha_j}) 
\leq C\sum_{j\leq s} d^{jn} \int_\Lambda\theta_s \cdot\omega_\Lambda^{d_\Lambda-j}\wedge  T^{j}_{\mathfrak{a}} 
\end{align*}
where $C$ is (another) constant that does not depend on $n$. The quantity
$\int_\Lambda\theta_s \cdot\omega_\Lambda^{d_\Lambda-j}\wedge  T^{j}_{\mathfrak{a}} $ is bounded by $\int_{K_{(d_\Lambda+1)\varepsilon}} \omega_\Lambda^{d_\Lambda-j}\wedge  T^{j}_{\mathfrak{a}}$. By hypothesis, we have 
 $T_\mathfrak{a}^{j}=0$ on $K_{(d_\Lambda+1)\varepsilon}$ for all $j\geq i$. In particular, $\int_{K_{(d_\Lambda+1)\varepsilon}} \omega_\Lambda^{d_\Lambda-j}\wedge  T^{j}_{\mathfrak{a}}=0$ for $j \geq i $. It follows that: 
\begin{align*}
\int_\Lambda\theta_1 \cdot\omega_\Lambda^{d_\Lambda-s}\wedge \bigwedge_{j=1}^k(\mathfrak{a}_{\ell_j})^*(\omega^{\alpha_j}) 
\leq C\sum_{j < i} d^{jn} \int_{K_{(d_\Lambda+1)\varepsilon}} \omega_\Lambda^{d_\Lambda-j}\wedge  T^{j}_{\mathfrak{a}}
\leq C'd^{(i-1)n} 
\end{align*}
  where $C'$ depends on (a neighborhood of) $K$ and $j$ but not on $n$. Summing over all $\alpha \in A_s$ and all $\ell \in L_n$ implies 
\begin{align*}
\mathrm{Vol}_\Omega\left(\Pi^{-1}_\Lambda(K_\varepsilon)\cap\Gamma_n\right) & \leq  C''\cdot n^{kq}\cdot d^{n(i-1)}
\end{align*}
again for some constant $C''>0$ which depends on $K$ but not on $n$. The inequality \eqref{eq:volume} gives
\[\frac{1}{n}\log N\leq (i-1)\log d +kq\frac{\log n}{n}+\frac{1}{n}\log C''-\frac{1}{n}\log(c\varepsilon^{2\dim(\Lambda)})\]
and the conclusion follows letting $n\to\infty$.
\end{proof}

\subsection{Metric bifurcation entropy}

\subsubsection{Metric entropy of a probability measure}
Pick a probability measure $\nu$ on $\Lambda$. Let $K\subset \Lambda$ be a compact set with $\nu(K)>0$. For any Borel set $X\subset K$, let
\[h_{\nu}(\hat{f},\mathfrak{a},K,X):=\lim_{\varepsilon\to0}\limsup_{n\to\infty}\frac{1}{n}\max\left\{\log\mathrm{Card}(E)\, ; \ E\subset X \ \text{is} \ (d_{\mathfrak{a},n},\varepsilon)\text{-separated}\right\}.\]
For $0<\kappa< \nu(K)$, we then let
\[h_{\nu}(\hat{f},\mathfrak{a},K,\kappa):=\inf\left\{h_{\nu}(\hat{f},\mathfrak{a},K,X)\, ; \ \nu(X)> \nu(K)- \kappa \right\}.\]
Finally, we define:
\[h_{\nu}(\hat{f},\mathfrak{a},K):= \sup_{\kappa\to 0} h_{\nu}(\hat{f},\mathfrak{a},K,\kappa)  \]
Observe that for any compact sets $K_1$ and $K_2$, it follows from our definition that \[h_{\nu}(\hat{f},\mathfrak{a},K_1\cup K_2 )= \max(h_{\nu}(\hat{f},\mathfrak{a},K_1 ),h_{\nu}(\hat{f},\mathfrak{a},K_2 )).\]
 We define the \emph{metric bifurcation entropy} of $\nu $ as
\[h_{\nu}(\hat{f},\mathfrak{a}):= \sup_K h_{\nu}(\hat{f},\mathfrak{a},K).\]
 It will be convenient, in what follows, to consider a small variation of the Bowen bifurcation distance defined as:
\[\tilde{d}_{\mathfrak{a},n}(\lambda,\lambda'):=\max\left(d_{\mathfrak{a},n}(\lambda,\lambda'), d_\Lambda(\lambda,\lambda')\right).\]
Notice that if we define accordingly $\tilde{h}_\mathfrak{a}(\hat{f},K)$, $\tilde{h}_{\nu}(\hat{f},\mathfrak{a},K,X)$,  $\tilde{h}_{\nu}(\hat{f},\mathfrak{a},K)$ and $\tilde{h}_{\nu}(\hat{f},\mathfrak{a})$ using  $\tilde{d}_{\mathfrak{a},n}$ instead of $d_{\mathfrak{a},n}$ then as $\tilde{d}_{\mathfrak{a},n} \geq d_{\mathfrak{a},n}$
we have that $\tilde{h}_\mathfrak{a}(\hat{f},K) \geq h_\mathfrak{a}(\hat{f},K)$,  $\tilde{h}_{\nu}(\hat{f},\mathfrak{a},K,X)\geq h_{\nu}(\hat{f},\mathfrak{a},K,X)$, $\tilde{h}_{\nu}(\hat{f},\mathfrak{a},K) \geq h_{\nu}(\hat{f},\mathfrak{a},K)$, $\tilde{h}_{\nu}(\hat{f},\mathfrak{a})\geq h_{\nu}(\hat{f},\mathfrak{a})$. On the other hand, by compacity  of $K \subset \Lambda$, a $d_\Lambda$-separated set in $K$ as bounded cardinality so the above inequalities are in fact equalities. In particular, we will still denote the above entropy as $h$  (and not $\tilde{h}$). 

Fix $\varepsilon>0$. For any integer $n$, any $\alpha>0$ and any $\gamma>0$, we let
\[X_{n,\gamma}(\nu,\alpha):=\left\{\lambda\in \Lambda\, ; \ \nu\left(B_{\tilde{d}_{\mathfrak{a},n}}(\lambda,\varepsilon)\right)\leq e^{-n(\alpha-\gamma)}\right\}.\]
The following is a volume argument.
\begin{proposition}\label{classical_volume_argument}
Fix $\alpha>0$ and a compact set $K\subset \Lambda$. Assume that for any $\kappa>0$ and any $\gamma>0$, there exists $n_0\geq1$ such that for all $n\geq n_0$,  $\nu(X_{n,\gamma}(\nu,\alpha) \cap K)\geq \nu(K)-\kappa >0$.
Then, $h_\nu(\hat{f},\mathfrak{a},K)\geq \alpha$.
\end{proposition}

\begin{proof} Let $X\subset K$ such that $\nu(X)>0$. 
Choose $\kappa>0$ small enough and pick $n_0\geq1$ so that $\nu(X_{n,\gamma}(\nu,\alpha) \cap K)\geq \nu(K)-\nu(X)/2$ and let $X':=X_{n,\gamma}(\nu,\alpha)\cap K$. By construction, $\nu(X')>0$.

Choose $\lambda_0\in X'$ and, recursively choose $\lambda_{k+1}\in X'\setminus\bigcup_{j\leq k}B_{\tilde{d}_{\mathfrak{a},n}}(\lambda_j,\varepsilon)$, which is possible as long as $X'\setminus\bigcup_{j\leq k}B_{\tilde{d}_{\mathfrak{a},n}}(\lambda_j,\varepsilon)\neq\varnothing$. Let $N\geq1$ be the cardinal of the set $E:=\{\lambda_j\, ; \ j\}$. Remark that $E$ is $(\tilde{d}_{\mathfrak{a},n},\varepsilon)$-separated and for all $k\leq N$,
\[\nu\left(\bigcup_{j=0}^{k-1}B_{\tilde{d}_{\mathfrak{a},n}}(\lambda_j,\varepsilon)\right)\leq ke^{-n(\alpha-\gamma)}\leq \nu(X').\]
In particular, this construction is possible, as long as $k\leq \nu(X')e^{n(\alpha-\gamma)}$. Whence $N\geq \nu(X')e^{n(\alpha-\gamma)}$ and
\[\frac{1}{n}\log(N)\geq\frac{1}{n}\log \nu(X')+\alpha-\gamma,\]
and making $n\to\infty$, we find $h_{\nu}(\hat{f},\mathfrak{a},K,X)\geq \alpha-\gamma$. Making $\gamma\to0$, we find $h_{\nu}(\hat{f},\mathfrak{a},K,X)\geq \alpha$. The result follows as $X$ is arbitrary.
\end{proof}

\subsubsection{The entropy of the bifurcation measure of a $k$-tuple of marked points}
We now come to the heart of the section.
Let $d_\Lambda:=\dim\Lambda$ and let $\mu_\mathfrak{a}$ be the probability measure which is proportional to $T_\mathfrak{a}^{d_\Lambda}$. 
The measure $\mu_\mathfrak{a}$ is the \emph{bifurcation measure} of the $k$-tuple $\mathfrak{a}=(a_1,\ldots,a_k)$ in $\Lambda$.

\begin{theorem}\label{minoration} Let $K$ be a compact set in $\Lambda$. 
 For any $\gamma>0$, there exists $n_0$ and $\varepsilon>0$ such that for all $\lambda_0 \in K$ and all $n\geq n_0$:
\begin{align*}
\mu_\mathfrak{a}\left(B_{\tilde{d}_{\mathfrak{a},n}}(\lambda_0,\varepsilon)\right)&\leq e^{-n d_\Lambda\log d + n \gamma }. 
\end{align*}
Consequently, for any compact set $K$ such that $\mu_\mathfrak{a}(K)>0$ then
\[ h_{\mu_\mathfrak{a}}(\hat{f},\mathfrak{a},K)= d_\Lambda \log d. \]
Thus $h_{\mu_\mathfrak{a}}(\hat{f},\mathfrak{a})= d_\Lambda \log d$.
\end{theorem}
\begin{proof}
Choose $\varepsilon>0$ and $\lambda_0\in  K$. Pick an integer $n\geq1$.  We let $X:= \Lambda \times(\P^q)^k$ and $\hat{f}_k:X \to X $ be the map defined by $\hat{f}_k(\lambda, (z_j)_{j\leq k}) = (\lambda,f_\lambda(z_1),\dots, f_\lambda(z_k))$. We consider the distance on $X$ defined by 
\[d( (\lambda, (z_j) ), (\lambda', (z'_j) )):= \max\left( d_\Lambda(\lambda,\lambda'),  \max_j d_{\P^q}(z_j,z'_j) \right) \]
where $d_\Lambda$ is the distance on $\Lambda$ induced by $\omega_\Lambda$ and $d_{\P^q}$ the distance on $\P^q$ induced by $\omega_{\P^q}$. We let $d_n$ denote the Bowen distance on $X$ associated to $d$:
\[ d_n( (\lambda, (z_j) ), (\lambda', (z'_j) )):= \max_{0\leq i\leq n-1} d( \hat{f}_k^i(\lambda, (z_j) ), \hat{f}_k^i(\lambda', (z'_j) )), \]
and we denote by $B_{d_n}( (\lambda, (z_j)), \varepsilon)$ the associated ball. 
With the notations of Section~\ref{background}, recall that (up to a multiplicative constant)
 \[\mu_\mathfrak{a}=  \sum_{j_1,\ldots,j_k\geq0 \atop j_1+\cdots+j_k=d_\Lambda }(\pi_\Lambda)_* \left( \bigwedge_{\ell=1}^k \pi_\ell^*\left(\widehat{T}^{j_\ell}\right) \wedge \left[\Gamma_{\mathfrak{a}}\right]\right).\]
 It is enough to prove the wanted estimate for each term of the sum. So from now on, fix a $k$-tuple $J:=(j_1,\ldots,j_k)$ with $j_1+\cdots+j_k=d_\Lambda $ and let 
 \[\mu^J_\mathfrak{a}:= (\pi_\Lambda)_* \left( \bigwedge_{\ell=1}^k \pi_\ell^*\left(\widehat{T}^{j_\ell}\right) \wedge \left[\Gamma_{\mathfrak{a}}\right]\right). \]
   Then
\begin{align*}
\mu^J_\mathfrak{a}\left(B_{\tilde{d}_{\mathfrak{a},n}}(\lambda_0,\varepsilon)\right)& = \int_{(\pi_\Lambda)^{-1}(B_{\tilde{d}_{\mathfrak{a},n}}(\lambda_0,\varepsilon))} \bigwedge_{\ell=1}^k \pi_\ell^*\left(\widehat{T}^{j_\ell}\right) \wedge \left[\Gamma_{\mathfrak{a}}\right] \\
&= \int_{B_{d_n} \left( \left(\lambda_0,  \mathfrak{a}\left(\lambda_0\right)\right), \varepsilon\right)}\bigwedge_{\ell=1}^k \pi_\ell^*\left(\widehat{T}^{j_\ell}\right) \wedge[\Gamma_{\mathfrak{a}}].
\end{align*}
Moreover, since $\widehat{T}=d^{-n+1}(\hat{f}^{n-1})^*\hat{\omega}+d^{-n+1}dd^c\widehat{u}_n$, where $(\widehat{u}_n)$ is a locally uniformly bounded family of continuous functions, letting  $u_{n,j}:=\widehat{u}_n\circ\pi_j$, we get
\begin{align*}
\mu^J_\mathfrak{a}\left(B_{\tilde{d}_{\mathfrak{a},n}}(\lambda_0,\varepsilon)\right)
&\leq d^{-(n-1)d_\Lambda}\int_{B_{d_n} \left( \left(\lambda_0, \mathfrak{a}\left(\lambda_0\right)\right), \varepsilon\right)}\bigwedge_{\ell=1}^k (\hat{f}_k^{n-1})^*( \Omega + dd^cu_{n,\ell})^{j_\ell}\wedge[\Gamma_{\mathfrak{a}}],
\end{align*}
where $\Omega:= \omega_\Lambda+\sum_{\ell=1}^k \pi_\ell^*(\omega_{\P^q})$. 
Let $\theta_n$ be the cut-off function of Lemma~\ref{dynamical_cutoff} in  $B_{d_n} \left( \left(\lambda_0, \mathfrak{a}\left(\lambda_0\right)\right), \varepsilon\right) $. Let $S:=\bigwedge_{\ell=2}^k (\hat{f}_k^{n-1})^*( \Omega + dd^cu_{n,\ell})^{j_\ell}\wedge[\Gamma_{\mathfrak{a}}]$. By Stokes formula and Lemma~\ref{dynamical_cutoff}, we deduce:
\begin{align*}
\mu^J_\mathfrak{a}\left(B_{\tilde{d}_{\mathfrak{a},n}}(\lambda_0,\varepsilon)\right)
&\leq d^{-(n-1)d_\Lambda}\int \theta_n \bigwedge_{\ell=1}^k (\hat{f}_k^{n-1})^*( \Omega + dd^cu_{n,1})^{j_1}\wedge[\Gamma_{\mathfrak{a}}] \\
 &\leq d^{-(n-1)d_\Lambda}\int \theta_n  (\hat{f}_k^{n-1})^*( \Omega) \wedge (\hat{f}_k^{n-1})^*( \Omega + dd^cu_{n,1})^{j_1-1}\wedge S+\\
&\quad   d^{-(n-1) d_\Lambda}\int u_{n,1}dd^c\theta_n \wedge (\hat{f}_k^{n-1})^*( \Omega + dd^cu_{n,1})^{j_1-1}\wedge S \\
&\leq \frac{Cn^2}{\varepsilon^2 d^{nd_\Lambda}} \int_{B_{d_n}\left(\left(\lambda_0, \mathfrak{a}\left(\lambda_0\right) \right)  ,2\varepsilon  \right)  }\left(\sum_{r=0}^{n-1} (\hat{f}_k^r)^*(\Omega)\right) \\
& \qquad \qquad \qquad\wedge (\hat{f}_k^{n-1})^*( \Omega + dd^cu_{n,1})^{j_1-1}\wedge S,
\end{align*}
where $C$ is a constant that depends only on the supremum of the $(u_{n,j})$ on the $2\varepsilon$-neighborhood of $\{\lambda_0\}\times (\P^q)^k$. We iterate the process with all the $j_1-1$ terms then for all the $\ell \leq k$. We deduce
\begin{align*}
\mu^J_\mathfrak{a}\left(B_{\tilde{d}_{\mathfrak{a},n}}(\lambda_0,\varepsilon)\right)
&\leq C\frac{n^{2d_\Lambda}}{d^{nd_\Lambda}} \int_{B_{d_n}\left( \left(\lambda_0, \mathfrak{a}\left(\lambda_0\right) \right), 2^{d_\Lambda} \varepsilon\right)}\sum_{0\leq m_1,\dots,m_{k}\leq n-1 } \bigwedge_{\ell=1}^k (\hat{f}_k^{m_\ell})^*( \Omega^{j_l})\wedge[\Gamma_{\mathfrak{a}}],
\end{align*}
where $C$ is a constant that depends only on $\varepsilon$ and the supremum of the $(u_{n,j})$ on the $2^{d_\Lambda}\varepsilon$-neighborhood of $\{\lambda_0\}\times (\P^q)^k$. Using Proposition~\ref{alaYomdin} (in a neighborhood of $K$) implies that there exists some constant $C'$ such that:
\begin{align*}
\mu^J_\mathfrak{a}\left(B_{\tilde{d}_{\mathfrak{a},n}}(\lambda_0,\varepsilon)\right)
&\leq C\cdot C'\cdot\frac{n^{2d_\Lambda}}{d^{nd_\Lambda}} e^{\gamma n},
\end{align*}
where $\gamma$ can be chose arbitrarily small by taking $\varepsilon$ small enough. This gives the wanted inequality in the theorem. Then, Proposition~\ref{classical_volume_argument} implies the inequality $h_\nu(\hat{f},\mathfrak{a},K)\geq d_\Lambda \log d$ so we have the equality by Theorem~\ref{alaGromov}. 
\end{proof}

Arguing similarly one proves that 
\begin{theorem}\label{entropy_current_bif} Pick an integer $1\leq j < d_\Lambda$. For any compact set $K$ such that  $T_\mathfrak{a}^j\wedge \Omega^{d_\Lambda-j} (K)>0$, we have
		\[ h_{T_\mathfrak{a}^j\wedge \Omega^{d_\Lambda-j}}(\hat{f},\mathfrak{a},K)\geq  j \log d .\]
	If furthermore $\mathrm{supp}(T_\mathfrak{a}^{j+1})\cap K=\varnothing$, then 	 
	\[ h_\mathfrak{a}(\hat{f},K)= j \log d .\]
\end{theorem}
The proof is the same than above, one proves first that for any $\lambda_0 \in K\Subset \Lambda$, then 
\[ T_\mathfrak{a}^j\wedge \Omega^{d_\Lambda-j}\left(B_{\tilde{d}_{\mathfrak{a},n}}(\lambda_0,\varepsilon)\right)\leq e^{-n j\log d + n \gamma }.\]
For that, we proceed as above though we can only replace $j$-terms $\widehat{T}$ by $d^{-n}(\hat{f}^n)^*\hat{\omega}+d^{-n}dd^c\widehat{u}_n$. We conclude by  Proposition~\ref{classical_volume_argument} and Theorem~\ref{alaGromov}. We also have the following parametric Brin-Katok formula for the bifurcation measure, similar to the dynamical one (\cite{brinkatok}, the ideas of our proof are similar).  
\begin{theorem}\label{brinkatok}
	For $\mu_\mathfrak{a}$-a.e. $\lambda$, one has:
	\[\lim_{\varepsilon\to0} \liminf_{n\to \infty} \frac{-1}{n} \log \mu_\mathfrak{a}\left(B_{\tilde{d}_{\mathfrak{a},n}}(\lambda,\varepsilon)\right)= \lim_{\varepsilon\to 0} \limsup_{n\to \infty} \frac{-1}{n} \log \mu_\mathfrak{a}\left(B_{\tilde{d}_{\mathfrak{a},n}}(\lambda,\varepsilon)\right)= d_\Lambda\log d.\]	
\end{theorem}
\begin{proof} Observe first that Theorem~\ref{minoration} above states that 
		\[\lim_{\varepsilon\to0} \liminf_{n\to \infty} \frac{-1}{n} \log \mu_\mathfrak{a}\left(B_{\tilde{d}_{\mathfrak{a},n}}(\lambda,\varepsilon)\right)\geq d_\Lambda\log d,\]
		for any $\lambda$ (not necessarily in the support of $\mu_\mathfrak{a}$). So all there is left to prove is that:
	\[\limsup_{n\to \infty} \frac{-1}{n} \log \mu_\mathfrak{a}\left(B_{\tilde{d}_{\mathfrak{a},n}}(\lambda,\varepsilon)\right)\leq d_\Lambda\log d	\]
	for $\mu_\mathfrak{a}$-a.e. $\lambda$. Take $\alpha > d_\Lambda\log d$ and let $\gamma \ll 1$ be such that $\alpha-\gamma > d_\Lambda \log d$. Since $h_\mathfrak{a}(\hat{f},\mathrm{supp}(\mu_\mathfrak{a}))= d_\Lambda \log d$, we know that for $\varepsilon$ small enough, there exists $n_0(\varepsilon)$ such that for all $n \geq n_0(\varepsilon)$, the cardinality of a $(n,\varepsilon)$-separated set is $\leq e^{n(d_\Lambda \log d +\gamma)}$.	
Consider the set:
\[X_{n}(\mu_\mathfrak{a},\alpha):=\left\{\lambda\in \Lambda\, ; \ \mu_\mathfrak{a}\left(B_{\tilde{d}_{\mathfrak{a},n}}(\lambda,\varepsilon)\right)\leq e^{-n\alpha}\right\}.\]
Take $\lambda_0$ in $X_{n}(\mu_\mathfrak{a},\alpha)$, then take inductively $\lambda_l \in  X_{n}(\mu_\mathfrak{a},\alpha)\backslash  \cup_{i\leq l-1} B_{\tilde{d}_{\mathfrak{a},n}}(\lambda_i,\varepsilon)$. This is possible as long as 
\[ \mu_\mathfrak{a}(X_{n}(\mu_\mathfrak{a},\alpha)) > l e^{-n\alpha}, \]
so, in particular, we can find a $(n,\varepsilon)$-separated set of cardinality $ \mu_\mathfrak{a}(X_{n}(\mu_\mathfrak{a},a))e^{n\alpha} $. By the bound of the entropy:
\[ \mu_\mathfrak{a}(X_{n}(\mu_\mathfrak{a},\alpha)) \leq e^{-n(\alpha-\gamma - d_\Lambda\log d)}.\]
Thus, as the rest of a convergent geometric series goes to $0$, we have:
\[ \mu_\mathfrak{a} \left(\limsup X_{n}(\mu_\mathfrak{a},\alpha)\right)=0.\]
As $\alpha$ is arbitrary, the result follows. 
\end{proof}

\begin{remark}\normalfont Using the same argument, we also have a  Brin-Katok formula for the measures $T_\mathfrak{a}^j\wedge \Omega^{d_\Lambda-j}$ on any compact sets $K$ such that $T_\mathfrak{a}^j\wedge \Omega^{d_\Lambda-j} (K)>0$ and $\mathrm{supp}(T_\mathfrak{a}^{j+1})\cap K=\varnothing$. 
\end{remark}

 \subsection{Bifurcation entropy of a holomorphic family of rational maps}\label{Section:bifurcationentropy}

Pick a holomorphic family $\hat{f}:\Lambda\times\P^1\to\Lambda\times\P^1$ of degree $d$ rational maps which is
critically marked, i.e. such that there exists $c_1,\ldots,c_{2d-2}:\Lambda\to\P^1$, holomorphic and such that for all $\lambda$, the points $c_1(\lambda),\ldots,c_{2d-2}(\lambda)$ describe all critical points of $f_{\lambda}$ counted with multiplicity.

\medskip

Let $\mathfrak{c}:=(c_1,\ldots,c_{2d-2})$. A theorem of DeMarco \cite{DeMarco1} states that the support of the closed positive $(1,1)$-current
$T_{\mathfrak{c}}$ coincides with the bifurcation locus in the classical sense of Ma\~n\'e-Sad-Sullivan \cite{MSS,Lyubich-unstable}. 
The bifurcation measure of the family $\hat{f}$ is the positive measure $\mu_\bif$ on $\Lambda$ defined as
\[\mu_{\bif}:=T_{\mathfrak{c}}^{\dim(\Lambda)}.\]
Note that, the formula \eqref{nointersection>q} reads as $T_{c_i}^2=0$ for any $1\leq i\leq 2d-2$, so that the $i$-th bifurcation current $T_\bif^i$ decomposes as
\[T_{\bif}^i=\sum_{ j_1,\ldots,j_i \atop \mathrm{distinct}}T_{c_{j_1}}\wedge\cdots\wedge T_{c_{j_i}}.\]
In particular, Theorem~\ref{Majoration} is just a reformulation of Theorem \ref{alaGromov} and Theorem \ref{Minoration} a reformulation of Theorem \ref{minoration}.

\medskip

Observe that we can easily generalize Theorems~\ref{Majoration} and ~\ref{minoration} to general families (i.e. non necessarily critically marked nor smooth). Indeed, pick a holomorphic family $\hat{f}:\Lambda\times\P^1\to\Lambda\times\P^1$ of degree $d$ rational maps. Take a finite branched cover $\pi:\tilde{\Lambda}\to\Lambda$ above the space of critically marked rational maps where $\tilde{\Lambda}$ is smooth, then the family $\tilde{f}$ defined by
\[\tilde{f}(\lambda,z)=(\lambda,f_{\pi(\lambda)}(z)), \ (\lambda,z)\in\tilde{\Lambda}\times \P^1,\]
is critically marked.

\begin{definition}
The \emph{bifurcation entropy} of the family $\hat{f}$ in a compact set $K\subset\Lambda$ is 
\[h_{\bif}(\hat{f},K):=h_{\mathfrak{c}}(\tilde{f},\pi^{-1}(K)).\]
\end{definition}

The above results apply then immediately.

\begin{question} \normalfont
	In \cite{InMuk}, the authors showed that the bifurcation locus of the anti-quadratic family: $(\lambda,z) \mapsto \bar{z}^2+\lambda$, the so-called Tricorn, contains undecorated real-analytic arcs at its boundary. In \cite{distribGV2}, we built a bifurcation measure which is supported by the closure of PCF parameters (so it does not see those arcs). It is easy to extend the notion of bifurcation entropy in that setting and it would be interesting to show that the bifurcations in the anti-quadratic family in the real-analytic arcs have no positive entropy and to show that the above bifurcation measure has maximal positive entropy.
\end{question}

\subsection{Application to point-wise dimension of the bifurcation measure} 
Let $f$ be a rational map in $\mathcal{M}_d$ (we identify $f$ with its class). We let $\mathcal{C}(f)$ denote its critical set. Assume $f$ is not a flexible Lattès map; for simplicity we also assume that $f$ has simple critical points and we let $F:\Lambda\times \P^1 \to \Lambda\times \P^1$ be a holomorphic family of rational maps that parametrizes a neighborhood of $f=f_0$ in $\mathcal{M}_d$. Up ot reducing $\Lambda$, we thus can follow holomorphically the critical points of $f$ in $\Lambda$.

 We make the following assumptions on $f$:
\begin{enumerate}
	\item  $f$  satisfies the Collet-Eckmann condition: $\forall c \in \mathcal{C}(f)$, $ \liminf |(f^n)'(f(c))|^{1/n}:=\exp(\underline{\chi}_c)>0$.  
\item $f$ satisfies the \emph{polynomial recurrence condition} of exponent $\beta$: There exists a constant $C>0$ such that $\forall c,c' \in \mathcal{C}(f)$, $\forall n \geq 1$, $\mathrm{dist}(f^n(c),c')\geq Cn^{-\beta}$. 
\end{enumerate}
 Misiurewicz maps provide many such examples. Following \cite{AGMV}, we see that the ball:
 \[\Omega_n:= \mathbb{B}\left(f, C\cdot \frac{1}{\max_c|(f^n)'(f(c))|} \right)\]
is sent into a $\varepsilon$-neighborhood of $(f^{n+1}(c_1), \dots, f^{n+1}(c_{2d-2}))$ by the map 
\[\lambda \to (f^{n+1}_\lambda(c_1(\lambda)), \dots, f^{n+1}_\lambda(c_{2d-2}(\lambda))),\] 
where $C$ is (another) constant that does not depend on $n$ and $(c_k(\lambda))$ denote the collection of marked critical points of $f_\lambda$. Let $ \limsup |(f^n)'(f(c))|^{1/n}:=\exp\overline{\chi}_c>1$. It follows that:
\[ \mathbb{B}\left(f, C\cdot  \frac{1}{\max_c\exp (n \overline{\chi}_c)} \right) \subset    B_{\tilde{d}_{\mathfrak{c},n}}(0,\varepsilon).\]
So taking the bifurcation measure $\mu_{\bif}$ of both sets and we have
 \[(\max_c \exp(\overline{\chi}_c))^{-n\overline{d}_{\mu_\bif}(f)}\lesssim\mu_{\bif}\left(\mathbb{B}\left(f, C\cdot  \frac{1}{ \max_c \exp (n\overline{\chi}_c)}\right) \right) \]
where $\overline{d}_{\mu_\bif}(f)$ denotes the (upper)-pointwise dimension of $\mu_\bif$ at $f$. From Theorem~\ref{minoration}, we have $   \mu_{\bif}(B_{\tilde{d}_{\mathfrak{c},n}}(0,\varepsilon))\lesssim e^{-n (2d-2)\log d
}$. So comparing the growth-rate, we deduce
\begin{equation}
\label{info_dim} (\max_c \exp \overline{\chi}_c)^{\overline{d}_{\mu_\bif}(f)}\geq d^{2d-2}
\end{equation}
which gives another proof of the second author's results \cite[Corollary 7.4]{Article1} in this more general context. 

\begin{remark} \normalfont
The work \cite{GaoShen} implies, for multimodal maps of the interval, that they are many examples satisfying the Collet-Eckmann assumption and the polynomial recurrence condition. By the above, it would be interesting to extend their result to rational maps. 
\end{remark}
Recall that, if $\mu_\Mand$ denotes the harmonic measure of the Mandelbrot set, it is the bifurcation measure of the family $f(\lambda,z)=(\lambda, z^2+\lambda)$ for which $0$ is the only marked critical point. A result of Graczyk and Swiatek \cite{graczykswiatek} states:
\begin{theorem}[\cite{graczykswiatek}]\label{graczykswiatek} For $\mu_\Mand$-almost every $\lambda$, the map $f_\lambda$ is Collet-Eckmann and
	\[ \underline{\chi}_0=\overline{\chi}_0=\lim_n  \frac{1}{n} \log \left| (f_\lambda^n)'(0) \right| =\log 2.\] 
\end{theorem}
As their proof relies crucially on fine properties of external rays, and on the fact that the parameter space is $\C$, it would be interesting to give a different proof, using the notion of bifurcation entropy which might also work in higher degree.

\section{Measure-theoretic entropy in several complex variables}\label{higher}
\subsection{Entropy of the Green measure of Hénon maps}
The purpose of this section is to give an alternate proof of the following result of Bedford and Smillie \cite[p-411, Theorem 4.4]{bedfordsmillie3}.
\begin{theorem}
	$h_\mu(f) = \log d$.
\end{theorem}
Recall that $f$ is a Hénon map of $\C^2$, that is a polynomial automorphism of degree $d>1$ of $\C^2$. The measure $\mu$ is defined as $\mu:= T^+\wedge T^-$ where $T^\pm$ is the Green current of $f^\pm$:
$T^\pm:= \lim_{n\to \infty} d^{-n} (f^{\pm n})^*(\omega)$ ($\omega$ is the Fubini-Study form on $\P^2$). As $I^+$, the indeterminacy point of $f$, is a super-attracting fixed point of $f^{-1}$, one can take an open set $U^+$ which is the complementary set of a neighborhood of $I^+$ such that $f(U^+)\subset U^+$. One constructs similarly $U^-$ and we can choose them so that the support of $\mu$ is relatively compact in $U:= U^+ \cap U^-$. We can write $  d^{-n+1} (f^{\pm (n-1)})^*(\omega)= T^\pm + dd^c u^\pm_n$ with the estimate  $  \|u^\pm_n\|_{\infty, U^\pm} \leq C d^{-n}$ (the constant $C$ depends on $U^\pm$ but not on $n$). As Bedford and Smillie did, observe that by Gromov's result and the variational principle,  $h_\mu(f) \leq \log d$ so all there is to prove is the reverse inequality.

 For that, they applied Yomdin's estimate on the dynamical ball $B_n(x,\varepsilon)$ for the measure $ d^{-n} (f^n)^*(\omega) \wedge \omega$ then 
using Misiurewicz' proof of the variational principle, they obtain the wanted lower bound of the entropy.\\

As briefly explained in the introduction, in here, the idea is to apply Yomdin's estimate on the dynamical ball $B_n(x,\varepsilon)$ directly for the measure $\mu$. This allows us to get rid of Misiurewicz' proof of the variational principle; in exchange, we need a precise control on the convergence towards the Green currents (we did not need in \cite{ThelinVigny1} in the general case of meromorphic maps). For that, we lift the different objects on the product space $\P^2 \times \P^2$, where the map $(f,f^{-1})$ acts, using the diagonal (those are ideas Dinh used to prove the exponential decay of correlations for Hénon maps \cite{dinhdecay}).

Let  $F$ be the birational map $F:=(f,f^{-1}):\P^2 \times \P^2  \to \P^2 \times \P^2   $, $\Pi_i$ be the projection to the $i$-th coordinate and $\Delta$ be the diagonal in  $\P^2 \times \P^2$. Write $\Omega= \Pi_1^*(\omega)+\Pi_2^*(\omega) $. Then $\mu= (\Pi_1)_*( \Pi_1^*(T^+)\wedge \Pi_2^*(T^-) \wedge [\Delta])$. If 
\[B_n(x,\varepsilon):= \{x'\in \P^2, \ \forall |k|\leq n-1, \   d(f^k(x'),f^k(x))< \varepsilon \} \]
is the two-sided dynamical ball ($d$ is the metric in $\P^2$ induced by the Fubini-Study form), then $B_n(x,\varepsilon) = \Pi_1( B^2_n((x,x), \varepsilon)\cap \Delta )$ where $B^2_n((x,x), \varepsilon)$ denotes the Bowen ball on $\P^2\times \P^2$ for the map $F$ with respect to the distance $\tilde{d}$ on $\P^2\times \P^2$ defined by 
$\tilde{d}((x,y),(x',y'))=\max (d(x,x'), (y,y'))$. In  other words:
\[ B^2_n((x,y), \varepsilon):= \{(x',y')\in \P^2, \ \forall k\leq n-1,\ \tilde{d}(F^k(x',y'), F^k(x,y))< \varepsilon\} .\]

We want an upper bound of $\mu(B_n(x,\delta))$ for $x \in \mathrm{supp}(\mu)$. Observe that the point $(x,x)\in U^+\times U^-$ and $F(U^+\times U^-)\subset U^+\times U^-$. Reducing $\varepsilon$, we can assume that $B^2_n((x,x), 4\varepsilon) \subset U^+\times U^-$. Using the cut-off function  $\theta_n$ of Lemma~\ref{dynamical_cutoff}, the estimates on the convergence toward the Green currents and Stokes formula, we have:
\begin{align*}
\mu(B_n(x,\varepsilon)) &= \int_{B^2_n((x,x), \varepsilon)}  \Pi_1^*(T^+)\wedge \Pi_2^*(T^-) \wedge [\Delta] \\
                  & =  \int_{B^2_n((x,x), \varepsilon)}   \Pi_1^*(d^{-n+1} (f^{n-1})^*(\omega)- dd^c u^+_n)\wedge\\
                  &\qquad \qquad \qquad \qquad\qquad   \Pi_2^*(d^{-n+1} (f^{-n+1})^*(\omega)- dd^c u^-_n)\wedge [\Delta] \\
                  &\leq \int \theta_n  \Pi_1^*(d^{-n+1} (f^{n-1})^*(\omega)- dd^c u^+_n)\wedge  \Pi_2^*(d^{-n+1} (f^{-n+1})^*(\omega)- dd^c u^-_n)\wedge [\Delta] \\
                  &\leq   \int \theta_n  \Pi_1^*(d^{-n+1}(f^{n-1})^*(\omega)) \wedge  \Pi_2^*(d^{-n+1} (f^{-n+1})^*(\omega)- dd^c u^-_n)\wedge [\Delta] \\ 
                  &\  \  + \int \Pi_1^*(u^+_n) dd^c \theta_n  \wedge  \Pi_2^*(d^{-n+1} (f^{-n+1})^*(\omega)- dd^c u^-_n)\wedge [\Delta] \\
                  &\leq  C\frac{n^2}{\varepsilon^2 d^n} \int_{B^2_n((x,x), 2 \varepsilon)} \sum_{k\leq n-1} (F^k)^*(\Omega) \wedge \Pi_2^*(d^{-n+1} (f^{-n+1})^*(\omega)- dd^c u^-_n)\wedge [\Delta], 
\end{align*}
where $C$ is a constant that does not depend on $n$. We proceed similarly for the term in $dd^c u^-_n$ using a dynamical cut-off function for $B^2_n((x,x), 2 \varepsilon)$ and we get:
\begin{align*}
\mu(B_n(x,\delta)) &\leq  C\frac{n^4}{\varepsilon^2 d^{2n}} \int_{B^2_n((x,x), 4 \varepsilon)} \sum_{k,l \leq n-1} (F^k)^*(\Omega) \wedge (F^{l})^*(\Omega)\wedge [\Delta], 
\end{align*}
where $C$ is again a constant that does not depend on $n$. We apply Proposition~\ref{alaYomdin} to the above term for $F$ and $\Delta$, so we have the bound:
\begin{align*}
\mu(B_n(x,\varepsilon)) &\leq  Ce^{\gamma n}, 
\end{align*}
for $\gamma$ arbitrarily small, reducing $\varepsilon$ if necessary. Arguing as in Proposition~\ref{classical_volume_argument} gives back Bedford-Smillie's theorem.

\begin{remark}  \normalfont
	  Proceeding as in the proof of Theorem~\ref{brinkatok} allows us to prove Brin-Katok formula for $\mu$ directly. In particular we do not need to use the ergodicity of $\mu$. It would be interesting and a priori difficult to get a speed in the convergence in Brin-Katok formula (for generic $x$). This raises the question of proving a quantitative Algebraic Lemma for holomorphic maps. \\
	The above proof also works in the case of holomorphic map or the so-called regular birational maps (\cite{dinhsibonyregular}).
\end{remark}

\subsection{Entropy for the trace measure of the Green currents}
   Let $f:\P^k \to \P^k$ be a holomorphic map of algebraic degree $d\geq 2$. Recall that its Green current $T_f$ is a positive closed current of bidegree $(1,1)$ and mass $1$ defined by:
   \[ T_f:= \lim_{n\to \infty} d^{-n} (f^n)^* \omega  \]   
   where $\omega$ is the Fubini-Study form on $\P^k$. We can write $T_f= d^{-n+1} (f^{n-1})^* \omega +dd^c u_n$ with $\|u_n\|_\infty=O(d^{-n})$. We can thus define the self-intersection of the current $T_f$ for $l\leq k$: $T_f^l:= T_f \wedge \dots \wedge T_f$ $l$ times. Its trace measure $\mu_l$ is then the well-defined probability measure:
   \[ \mu_l := T_f^l\wedge \omega^{k-l}. \]
When $k=l$, $\mu_k$ is known to be the (unique) ergodic measure of maximal entropy $k\log d$ \cite{briendduval} so we shall assume that $l<k$. Our aim is to show similar results for $\mu_l$ in $\mathrm{supp} (T_l^k)$. Since $\mu_l$ is not invariant, we need to give a meaning to its entropy in term of the asymptotic cardinality of $(n,\varepsilon)$-separated sets. The Bowen distance $d_n$ we consider here is the classical distance on $\P^k$: $d_n(x,y):= \max_{j\leq n-1} d(f^j(x),f^j(y))$.
\begin{definition}
Let $K$ be a compact set in $\P^k$ and let $\nu$ be a probability measure on $\P^k$.
For $\kappa>0$, we consider:
\[ h_\nu(f,K,\kappa):= \inf_{\Lambda \subset K, \ \nu(K\backslash\Lambda) < \kappa } \sup_{\varepsilon>0} \limsup_{n \to \infty } \frac{1}{n} \log \max \left\{ \mathrm{card}(E), \ E\subset \Lambda \ \mathrm{is} \  (n,\varepsilon)-\mathrm{separated} \right\}.
\]
 We define the \emph{entropy $h_\nu(f,K)$ of $f$ on $K$} as the quantity 
\[ h_\nu(f,K):=\sup_{\kappa\to 0} h_\nu(f,K,\kappa).
 \]
 We also define the \emph{entropy $h_\nu(f)$ of $f$} as $h_\nu(f,\P^k)$.
\end{definition}    
\begin{theorem}For any $\gamma>0$, there exists an integer $n_0$ and  $\varepsilon>0$  such that for any $x\in \P^k$ and any $n\geq n_0$
	\[\mu_l(B_n(x,\varepsilon)) \leq e^{- n l\log d + n\gamma } .\]
	 In particular, for any compact set $K\subset \mathrm{supp}(T^l_f)\backslash  \mathrm{supp}(T^{l+1}_f)$ such that $\mu_l(K)>0$, we have $h_{\mu_l}(f,K)=l\log d$.	 
\end{theorem}
The proof of the first part is similar to the one of Theorem~\ref{minoration} above so we omit it (replace $T_f$ by $d^{-n+1} (f^{n-1})^* \omega +dd^c u_n$, apply Stokes to the dynamical cut-off functions and conclude with Proposition~\ref{alaYomdin}). The second part then follows directly from the result of the first author (\cite{dethelin_mathann}, see also \cite{dinh_JGA}) who showed that the topological entropy is $\leq l\log d$ outside the support of $T^{l+1}_f$. \\

Finally, we can also consider the map $f\to h_{\mu_l}(f)$. 
\begin{proposition}
Let $f$ be an endomorphism of $\P^k$ of algebraic degree $d\geq2$. Then $h_{\mu_l}(f)\geq l \log d$.	
Furthermore, for any $\alpha \in [l \log d, k \log d]$, there exists an endomorphism $f$ such that $h_{\mu_l}(f)= \alpha$. 	
\end{proposition}
\begin{proof}
	We start with an easy example where $h_{\mu_l}(f)> l \log d$. Take for that $f$ a Lattès maps of $\P^k$. Then $\mu_l$ is absolutely continuous with respect to the Fubini-Study measure on $\P^k$ and with respect to $\mu_k$ (this is even a characterization of Lattes \cite{bertelootdupontlattes}). Now, any set of positive $\mu_l$ measure contains a set of positive $\mu_k$ measure so it gives an entropy $k\log d$.  	
	
	For simplicity, let us restrict ourselves to the case where $l=1$, $k=2$. We recall for that a construction in \cite[p. 603]{dujardindirection}. Take $h$ a rational map of $\P^1$ of degree $d$ which admits an ergodic measure $\nu$ absolutely continuous with respect to the Lebesgue measure of $\P^1$ of entropy $0< \beta <\log d$ (they are many such examples). Let $\mu_h$ denote the measure of maximal entropy $\log d$ of $h$. 
		
	Consider now the map $\hat{f}:=(h,h)$ acting on $\P^1 \times \P^1$. Taking the quotient of $\P^1 \times \P^1$ by $(z,w)\equiv (w,z)$, we get a holomorphic map $f$ of degree $d$ on $\P^2$. Let $\pi_i$ denote the projection to the $i$-th factor of $\P^1\times \P^1$. The trace measure of the Green current of $\hat{f}$ is $\pi_1^*(\mu_h) \wedge \pi_2^*(\omega) + \pi_2^*(\mu_h) \wedge \pi_1^*(\omega) $ which is absolutely continuous with respect to $\pi_1^*(\mu_h) \wedge \pi_2^*(\nu) + \pi_2^*(\mu_h) \wedge \pi_1^*(\nu) $. In particular, its entropy is $\log d + \beta$. Descending to $f$, we deduce that the entropy of the trace measure of the Green current is $\log d + \beta$.
	\end{proof}

\begin{remark} \normalfont To study the entropy on $\supp(T_f)\backslash \supp(\mu_f)$ for an endomorphism of $\P^2$, the first author defined saddle measures, under general assumptions \cite{thelinselle}.	The advantage of that approach is that we deal with nice ergodic, invariant measures but the support of such measure can be much smaller than $\supp(T_f)$.  
\end{remark}

\bibliographystyle{short}

\end{document}